\documentclass[12pt]{article}
\usepackage[utf8]{inputenc}
\usepackage [a4paper,left=2.5cm,right=2.5cm,top=2.5cm]{geometry}
\usepackage{amssymb}
\usepackage{amsfonts, color}
\usepackage{graphicx}
\usepackage{enumerate}
\usepackage{amsmath}%
\usepackage{hyperref}
\usepackage{epsfig}

\providecommand{\U}[1]{\protect\rule{.1in}{.1in}}
\newtheorem{theorem}{Theorem}[section]

\newtheorem{corollary}[theorem]{Corollary}

\newtheorem{definition}[theorem]{Definition}
\newtheorem{example}[theorem]{Example}

\newtheorem{lemma}[theorem]{Lemma}

\newtheorem{proposition}[theorem]{Proposition}
\newtheorem{remark}[theorem]{Remark}

\newenvironment{proof}[1][Proof]{\textbf{#1.} }{\ \rule{0.5em}{0.5em}}

\usepackage{verbatim}

\newcommand{\dive}{{\rm div}}

\newcommand{\Hund}{{\rm \text{\scriptsize{H-}}und}}

\newcommand{\Hhund}{{\rm \text{\scriptsize{H-}}hund}}

\newcommand{\Hscap}{{\rm \text{\scriptsize{H-}}cap}}

\newcommand{\Graph}{{\rm Graph}}

\newcommand{\rr}{\mathbb{R}}
\newcommand{\HH}{\mathbb{H}^2}

\newcommand{\hhr}{\mathbb{H}^2\times \mathbb{R}}
\newcommand{\hh}{\mathbb{H}^2}

\newcommand{\brg}{\partial_{g}}
\newcommand{\brx}{\partial_{\times}}
\newcommand{\HM}{0\leqslant H\leqslant1/2}

\newcommand{\dern}{\frac{2H}{\sqrt{1-4H^2}}}

\title{Geodesic boundary of constant mean curvature surfaces in $\mathbb{H}^2 \times \mathbb{R}$}
\author{Felix Nieto Cacais\footnote{Partially supported by CNPq, 141269/2019-7}, Miriam Telichevesky\footnote{Partially supported by CNPq, 304271/2016-0}}

\begin{document}

	\maketitle

\begin{abstract}
    It is provided a classification of curves in the geodesic boundary $\brg(\hhr)$ that are the boundary of properly embedded CMC-$H$ surfaces, $0\le H<1/2$ of the product manifold $\hhr$. It is also defined a notion of ``asymptotic height'' of a class of unbounded functions defined in $\hh$ whose graphs have CMC $0<H<1/2$, giving sense for Dirichlet problems at infinity for those graphs and proving that they are solvable under usual assumptions.
\end{abstract}

	
	\section{Introduction}

	In this paper we study curves in the geodesic boundary of $\hh\times\mathbb{R}$ which are the geodesic boundary of properly embedded surfaces in $\hhr$ with constant mean curvature between $0$ and $1/2$. The \textit{geodesic boundary}, also know as the \textit{asymptotic boundary}, of a Hadamard manifold $M$ is the set $\brg M$ of equivalence classes of geodesic rays that stay at a finite distance. The set $\overline{M}:=M\cup (\brg M)$ is the so-called \textit{geodesic compactification} of $M$; if endowed with the cone topology, it is a compact topological space. If you are not familiar with this notion, see Section \ref{sec:geodesic_compactification} for more details.

Maybe one of the most interesting Hadamard manifolds where it is very nice to work with the geodesic compactification is the hyperbolic plane $\hh$. In particular, it makes sense to prescribe some continuous function $\Psi:\brg \hh \to \mathbb{R}$ and ask if there is any solution for some Dirichlet problem (for Laplace's equation or for minimal graph equation, for instance, being the second one our main interest) that atains $\Psi$ at infinity.

In this direction, we metion the very important result of B. Nelli and H. Rosenberg \cite{NR}:

\begin{theorem}[Theorem 4 of \cite{NR}]\label{theo:theoNR}
 Let $\Gamma$ be a continuous Jordan curve in $(\brg \hh)\times \rr$, that is a vertical graph. Then, there exists a minimal vertical graph on $\hh$ having $\Gamma$ as asymptotic boundary. The graph is unique.
\end{theorem}

Also, in 2008 R. Sa Earp and E. Toubiana \cite{SET} built a family of curves on $\brg(\hh)\times\rr$ which are also boundary of minimal surfaces:
 \bigskip

 \begin{theorem}[Proposition 2.1 of \cite{SET}]\label{theo:propSET}
     Let be $a,b\in\overline{\rr}$, $\widehat{q_1,q_2}$ equatorial arc in $\brg\HH$ and $\sigma$ the boundary of the rectangle $\widehat{q_1,q_2}\times[a,b]$ with $q_1\neq q_2$, $b-a\geqslant\pi$.  Then there is a properly embedded minimal surface $\Sigma$ such that its geodesic boundary is $\sigma$.
 \end{theorem}

	Another result in this sense is given in the paper of B. Kloeckner and R. Mazzeo in \cite{KLM}:

\begin{theorem}[Theorem 4.1 of \cite{KLM}]
    Let $\sigma \subset \brg(\hh)\times\rr$ be  a  curve,  and  assume  that  every $p\in\brg(\hh)\times\rr\setminus \sigma$ can be separated from $\sigma$ 		by the boundary $\brg\Sigma_p$ of a properly immersed minimal surface $\Sigma_p$.Then $\sigma$ is minimally fillable.
\end{theorem}

	In the same paper, the authors give a new notion of boundary of minimal graphs, working with the geodesic boundary $\brg(\hhr)$ (which we will also kindly call ``the onion boundary'', see Figure \ref{fig:imagen1} bellow), what makes sense, since $\hhr$ is itself a Hadamard manifold. But in contrast of the product $(\brg \hh) \times \rr$, where different minimal surfaces could have different boundaries, all minimal surfaces in $\hhr$ mentioned in the previous theorems have the \textit{same} curve in $\brg(\hhr)$ as geodesic boundary. More precisely, if a minimal surface is contained between two slices $\hh\times\{0\}$ and $\hh\times\{L\}$, $L>0$, then its geodesic boundary in $\brg(\hhr)$ must be the so-called ``equator''. It is the case, for instance, of the minimal surfaces of Theorems \ref{theo:theoNR} and \ref{theo:propSET} above.
 
The Corollary 5.2 of B. Kloeckner and R. Mazzeo in \cite{KLM} provides a beautiful classification of possible curves (union of disjoint Jordan curves) $\sigma \in \brg(\hhr)$ that are \textit{minimally fillable}, that is, that are the geodesic boundary of some properly embedded minimal surface of $\hhr$. In the same paper they leave open the questions of classification of minimally fillable in $\brg(\hh)\times\rr$ or higher-dimensional symmetric spaces. Our work aims primarily to extend the classification of the result in another context: for properly embedded surfaces with constant mean curvature $H\in [0,1/2]$; in this sense we define $H-$fillable curves in $\brg(\hhr)$ (Definition \ref{def:Hfillablecurve})) and give the corresponding classification result in Corollary \ref{cor1}. We do not present it here because some notation must be first constructed. 
\

As one of the consequences of Corollary \ref{cor1}, we obtain that all CMC$-H$ entire graphs over $\hh$, $H\in [0,1/2]$ have their geodesic boundary in $\brg(\hhr)$ determined only by $H$ (up to orientation). We should point out that it does not happen with minimal graphs when considering $(\brg\hh)\times \mathbb{R}$, as in Theorem \ref{theo:theoNR} of \cite{NR} above, where, by uniqueness, different functions \textit{must have different boundary data}. The second great contribution of the present work is to obtain some kind of blow-up of $\brg(\hhr)$ at the geodesic boundary of CMC-$H$ entire (and well-behavied at infinity) graphs, for $H\in[0,1/2)$ previously fixed. It turns out that there is a natural place for prescribing Dirichlet boundary data for CMC-$H$ complete graphs with continuous normal vector field up to infinity, even knowing that those graphs correspond to functions that diverge at infinity. This new concept of boundary, defined and explored in Section \ref{sec:B_H}, allows us to distinguish different CMC-$H$ graphs that have the same geodesic boundary, extending the situation that occurs in the minimal case in $(\brg \hh)\times \rr$.

	\section{Geodesic compactification}\label{sec:geodesic_compactification}

	
	 Let be $M$ a Hadamard manifold. As in the work of Eberlein and O'Neil \cite{EO}, the geodesic rays are geodesics $\gamma:[0,+\infty]\rightarrow M$ with $|\gamma'(t)|=1$. Two geodesic rays $\gamma_1,\gamma_2$ are asymptotic if $\exists C\geqslant 0$ such that
	$$d(\gamma_1,\gamma_2)\leqslant C,\forall t\in[0,+\infty).$$

	The relation $\sim$ called asymptotic is an equivalence relation. The asymptotic boundary, or geodesic boundary (as we call in this work),	of $M$ is defined as 
	$$\brg M:=\{geodesic\,\,\,rays\}/\sim,$$ first introduced by Eberlein and O'Neil \cite{EO}.
	
	The equivalence class of a ray $\gamma$ will be denoted as $[\gamma]$. We remark that fixing $p\in M$ there is a one-to-one correspondence between the unitary sphere $S_{p}\subset T_{p}M$ and $\brg M$, given $v\in S_{p}\mapsto[\gamma_{v}]\in\brg M$, where $\gamma_{v}(t)=exp_{p}(tv)$.

	Also in \cite{EO} it is defined the topological compactification of a Hadamard manifold including its geodesic boundary. For this purpose there are defined the ``truncated cones'' $T_{p}(v,\theta,R):=C_{p}(v,\theta)\backslash\overline{B(p,R)}$. Here $C_{p}(v,\theta)=\{exp_{p}(tw)|\sphericalangle(w,v)<\theta,\,t\ge 0\}$ is the cone with vertices at $p\in M$, opening angle $\theta\in(0,\pi)$ and axis $v\in T_{p}M$, and $\overline{B(p,R)}$ is the ball with a centered in $p$ radius $R>0$. The open balls in $M$ together with the truncated cones $T_{p}(v,\theta,R)$ form a topological basis for $\overline{M}^g:=M\cup\brg M$. With respect to the cone topology $\overline{M}^g$ is a compact topological space, named by {\em geodesic compactification} of $M$. We remark that the geodesic boundary $\brg M$ has the topological subspace topology. 
	

	
	In this text we will focus on the Hadamard manifold $\hhr$. First of all, let us start precisely describing the geodesic compactification $\overline{\hhr}^g=(\hhr)\cup\brg(\hhr).$ The geodesic rays in the product manifold $\hhr$ have one of the following forms:
	
\begin{description}
  \item[(i)] There are geodesic rays $\gamma_1:[0,+\infty)\to \hh$ and $\gamma_2:[0,+\infty) \to \rr$ such that $\gamma(t) = (\gamma_1(t),\gamma_2(t))$.
  \item[(ii)] There exist a point $a\in \hh$ and a geodesic ray $\gamma_2:[0,+\infty) \to \mathbb{R}$ such that $\gamma(t)=(a,\gamma_2(t))$. In this case, $\gamma_2$ is explicitely given by $\gamma_2(t)=\epsilon t + b$, where $\epsilon = +1$ or $-1$ and $b\in \rr$ is a given constant. These geodesic rays are then the vertical half-lines.
  \item[(iii)] There exist a geodesic ray $\gamma_1:[0,+\infty) \to \hh$ and a constant $b\in \rr$ such that $\gamma(t)=(\gamma_1(t), b).$ These geodesic rays are the horizontal ones, contained in the slice $\hh\times\{b\}$.
\end{description}

	From now on, we denote by $p^-$ and $p^+$ the equivalence classes of the geodesics rays whose trace are of the form $\{a\}\times \rr^-$ and $\{a\}\times\rr^+$, respectively, calling these points as {\em poles}. We also use the word {\em equator} to the set of points in $\partial_g(\hhr)$ that are de equivalence classes of geodesics of constant height; notice that there is a natural identification of this set with the geodesic boundary of $\hh$; for abuse of notation we will denote the equator by $\partial_g \hh$.

	An important remark is that geodesics in $\rr$ are of the form
	$\beta =b+mt$, so equivalence classes can be identified with geodesic slope, that is $[\beta]\equiv m$, so any element of $\brg(\hhr)\backslash\{p^+,p^-\}$ can be noted as $([\gamma],m)$, where $\gamma$ is geodesic ray in $\HH$ and $m$ slope of a geodesic at $\rr$, the poles can be denoted as $p^+=[(p,+\infty)],p^-=[(p,-\infty)]$ see Figure \ref{fig:imagen1}.

	\begin{figure}[h]
		\centering
		\includegraphics[width=0.3\linewidth]{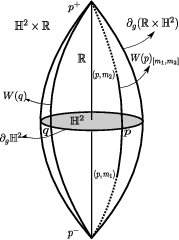}
		\caption{Geodesic boundary $\brg(\hhr)$}
		\label{fig:imagen1}
	\end{figure}

	\bigskip
	\bigskip
	
	\begin{remark}[Notation for points in the geodesic boundary]
		Although $\brg(\hhr)$ is not a product, we from now on will denote their points by ordered pairs because it inherit, in some sense, the notion of product of $\hhr$. The notation $(m,p)\in \partial_g\hhr$ will be used for the equivalence class of the geodesic ray $t\in[0,+\infty) \mapsto (\gamma_1(t),mt)$, where $\gamma_1$ is a unitary geodesic ray in $\hh$ such that $[\gamma_1]=p$ in $\partial_g\hh$.  
	\end{remark}

	The ``vertical lines'' in $\brg\hhr$ connecting $p^+$ to $p^-$ going through $p\in\brg\HH$ are called Weyl Chambers and will be noted as $W(p)$, that is, given $p\in \brg\hh$,
	$$W(p) = \{[(\gamma_1,\gamma_2)]|[\gamma_1]=p\in \partial_g\hh\}\cup \{p^-\}\cup\{p^+\}.$$ 
	Weyl Chambers’s open and closed interval will be used more frequently throughout the text, in which case the notation is $W_{[m_1,m_2]}(p)$(open and closed depends on context) where $m_1$ e $m_2$ is the initial and final slope in the second factor of each geodesic ray (possibly being $\pm$ infinity). Weyl Chamber is a concept that comes from symmetric spaces and the notation used in this text was introduced in \cite{KL}.

	\begin{remark}
		The product $\overline{\HH}\times\overline{\rr}$ is compact with the product topology, in which case the boundary would be the set $\brg\HH\times\overline{\rr}$ thus getting another compactification for $\hhr$ called  product compactification and the notation will be $\overline{\hhr}^{\times},\brx(\hhr)$. The product boundary on the geodesic boundary is reduced to the equator $\brg\HH$ and the two poles $\{p^+,p^-\}$, more details see \cite{KLM}.	
	\end{remark}

	\section{Invariant surfaces and $H-$fillable curves}\label{sec:invariant_surfaces}

	\subsection{Conventions and preliminary results}
	In \cite{KLM} the authors have defined the minimally fillable curves, in this text we present a definition of $H-$fillable curves for $\HM$. Since for $H>0$ we need to take care about orientation, we also define what is a surface that points ``inwards'' to $\hhr$. Before the definition, we will establish that throughout the following sections we will assume that {\em the mean curvature vector of the CMC surfaces extends continuously up to the boundary}.
	
	\begin{definition}\label{def:Hfillablecurve}
		A curve $\sigma\subset\brg\hhr$ is said to be $H-$fillable if it is the geodesic boundary of properly embedded surface in $\hhr$ with constant mean curvature $H$, where $0\leq H\leq 1/2$ and curve means Jordan curve or union of disjoint Jordan curves.
	\end{definition}

	\begin{definition}[Orientation of a CMC surface]\label{definition-orientation}
	
	Fix any point $o\in \hh$. Let $\Sigma$ be a CMC surface with normal vector field given by $\eta$. If $(q,m)\in \brg \Sigma$, we say that $\Sigma$ is oriented pointing inwards at $(q,m)$ if there exists $V\subset\overline{\hhr}$ a neighborhood of $(q,m)$ such that for any $(x,z)\in \Sigma\cap V$, it holds that $\left\langle \eta(x,z), \gamma_x'(s(x))\right\rangle \le 0$, where $\gamma_x:[0,+\infty)\to \hh$ is the unique unitary geodesic ray such that $\gamma_x(0)=o$ and $\gamma_x(s(x))=x$.
	    
	\end{definition}
	
	This means that the $\hh$ component of the normal vector field $\eta$ of $\Sigma$ points inwards to $\hh$ at infinity.
	

\begin{lemma}\label{LemaOrienta}
Suppose that $\Sigma$ is a CMC surface in $\hhr$ that points inwards at $(q,m)\in \brg \Sigma$ and that $\brg \Sigma$ is transversal to $W(q)$. If $m>0$, then there exists a neighborhood $V\in \overline{\hhr}$ of $(q,m)$ such that $\left\langle \eta(p), \partial_z \right\rangle > 0$ for any $p\in \Sigma \cap V$, where $\partial_z$ stands for the standard tangent vector field of the factor $\rr$ in $\hhr$. 
\end{lemma}

\begin{proof}
  Let $(p,m)\in\brg\Sigma$ such that $\brg\Sigma$, with $m>0$ is transversal to $W(q)$ in $(p,m)$. From the convergence of $\eta$ together with the transversality we have that there is a neighborhood $V$ such that $\Sigma\cap V$ is a graph over some unbounded domain $\Omega \subset \hh$. Let $v:\Omega\rightarrow\rr$ such that $\Sigma\cap\ V=Graph(v)$, then the normal vector is $\eta=\frac{(\nabla v,-1)} {\sqrt{1+|\nabla v|^2}}$ or $\eta=\frac{(-\nabla v,1)}{\sqrt{1+|\nabla v|^2}}$.  
  
  Suppose $\eta=\frac{(\nabla v,-1)}{\sqrt{1+|\nabla v|^2}}$. If $\langle \eta, \gamma'\rangle\le0$, we have $\langle \nabla v, \gamma'\rangle\le0$, so $(v\circ\gamma)'(s)\ge0$ , for $s$ large. Thus, $\Sigma\cap\gamma_v$ is a non-increasing function, which contradicts the fact $([\gamma], m)\in\brg\Sigma$, $m>0$, because $(v\circ \gamma)'\to m$. Therefore $\eta=\frac{(\nabla v,1)}{\sqrt{1+|\nabla v|^2}}$, that is, $\eta$ is pointing upwards. The same reasoning can be done for $m<0$, assuming that $\eta=\frac{(-\nabla v,1)}{\sqrt{1+|\nabla v|^2}}$ and $ \langle \eta, \gamma'\rangle\ge0$.
\end{proof}

\begin{remark}
  If $\brg\Sigma$ intersects $W(p)$ transversally at $(p,0)$ from the convergence and continuity of $\eta$ we have that $v$ is a bounded function and $\eta$ is pointing upwards or downwards in a neighbourhood $V$ of $(p,0)$. Indeed, if $\Sigma$ converges to $(p,0)$ with a horizontal normal vector $\eta$, it means that $\lim_{x\mapsto p}\frac{\nabla v}{\sqrt {1+|\nabla v|^2}}=0$; therefore $\lim_{x\mapsto p}|\nabla v|=+\infty$ which means that $v$ is not bounded, contradicting the fact that $\Graph(v)$ converges to $(p,0)$.
 \end{remark}

\begin{figure}[h]
    \centering
    \includegraphics[width=0.33\linewidth]{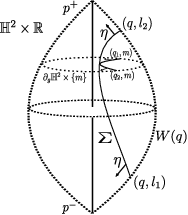}
    \caption{Orientation}
    \label{fig:my_label}
\end{figure}

	\subsection{Useful examples}
	
	In the following we present the graphs of some functions in $\hh$ and their respective geodesic boundaries, that are particular cases of well-known invariant CMC surfaces of $\hhr$. The litterature about those invariant surfaces is very extensive and includes the works of R. Sa Earp \cite{STP}, R. Sa Earp and E. Toubiana \cite{STM} and \cite{SET}, B. Nelli, R. Sa Earp, W. Santos and E. Toubiana \cite{BRWE}, G. Citti and C. Senni \cite{{CS}} and P. Klaser, R. Soares and the second author \cite{KST}; the case $H=1/2$, in particular, appear in \cite{MBR} and in \cite{CH}. We remark that some of the results of this Section are also presented in the previously mentioned works, but for sake of completeness, to fix some notation and to emphasize the desired properties, we write them here. 
 
	First note that if $\Omega\subset \hh$ and $u:\Omega \to \mathbb{R}$ is a $C^2$ function such that its graph $$Gr(u) = \{(x,u(x))|x\in \Omega\}$$ is a constant mean curvature $H$ surface in $\hhr$ (with mean curvature vector pointing {\em upwards}
), then it is well known that $u$ satisfies the elliptic quasilinear PDE
	\begin{equation}\label{cmc-equation}
		Q_H(u):=\dive\left(\frac{\nabla u}{\sqrt{1+|\nabla u|^2}}\right)- 2H = 0
	\end{equation}
	where $\dive$ and $\nabla$ are those ones defined by the riemannian metric of $\hh$.
	
	It is important to point out that in most cases in literature the mean curvature vector of a graph is defined to point downwards and hence there is a $+$ sign in front of $2H$, instead of the minus sign in equation \eqref{cmc-equation}. Here we have chosen the opposite orientation to avoid a lot of $-$ signs in the constructed functions. As it will be clear in the following, the minus sign correspond just to multiply by $-1$ each of those functions.

	Choosing functions that are defined in terms of the distance function $s$ of $\hh$ to a circle (that may be assumed to be a point, the circle of radius $0$), a horocircle or an hypercicle it is possible to construct surfaces that are invariant by a $1-$parameter of elliptic, parabolic or hyperbolic isometries, respectively, with respect to the $\hh$ factor, surfaces initially introduced in \cite{STM},\cite{BRWE} and \cite{STP}. More precisely:
	\begin{description}
		\item[Rotational graphs:] If $o\in \hh$ is a point and $r\ge 0$, we chose $\Omega = \{x\in \hh | d(x,o) \ge r\} = \hh \setminus B_r(o).$ In this case $s: \Omega \to \mathbb{R}$ is the distance to the circle of radius $r$ centered in $o$, $\partial B_r(o)$, and satisfies $\Delta s(x) = \coth(r+s(x))$ for all $x\in \Omega$.
		\item[Parabolic graphs:] If $p\in \brg \hh$ is a point in the geodesic boundary of $H^2$ and $D$ is a fixed horodisc having $p$ as boundary point, then choosing $\Omega = \hh\setminus D$ one has $s$ the distance to $\partial D$, a horocicle passing through $p$, and $s$ satisfies $\Delta s(x) = 1$.
		\item[Hyperbolic graphs:] Let us chose (the image of a) geodesic $\gamma$. Here it is convenient to consider the ``distance $s$ to $\gamma$ with sign'', in the following sense: $\gamma$ splits $\hh$ in two connected components $\Omega^-$ and $\Omega^+$, one of them, say $\Omega^+$, containing in its geodesic boundary a previously chosen point $q$. We say that $s(x)$ to $\gamma$ is positive if $x\in \Omega^+$ and $s(x)<0$ if $x\in \Omega^-$. It then holds that $\Delta s(x) = \tanh(s(x)).$

	\end{description}
	In all of the three cases, we chose $u=\phi \circ s: \Omega \to \mathbb{R}$ and therefore the PDE \eqref{cmc-equation} gives rise to the ODE
	
	\begin{equation} \label{CMC-edo}
		\left(\frac{\phi'(s)}{\sqrt{1+\phi'(s)^2}}\right)'+ \left(\frac{\phi'(s)}{\sqrt{1+\phi'(s)^2}}\right) \Delta s -2H=0.
	\end{equation}
	
	In the following we briefly describe some of the obtained invariant surfaces (see more details in \cite{KST}) and also we emphasize the geodesic boundary of each graphs. Such geodesic boundaries provide the first -- and almost all -- examples of $H-$fillable curves.

	\subsubsection{The unduloids $\Hund_r$ and the caps $\Hscap$}

	Let us start with the rotational surfaces.

The first rotational surface we present is the  $\Hund_r$. We consider the distance function to a circle of positive radius $r$ centered at $o\in \hh$. Solving \eqref{CMC-edo} with initial condition $\phi'(0)=+\infty$, we find the general expression $\phi(s)=\Hund_r(s)$:	
	\begin{equation}\label{eq-hnod}
		\Hund_r(s):=\displaystyle\int_0^s \frac{2H(\cosh(r+t)-\cosh(r))+\sinh(r)}{\sqrt{\sinh^2(r+t)-(2H(\cosh(r+t)-\cosh(r))+\sinh(r))^2}}dt,
	\end{equation}
	defined in $[0, +\infty)$.
	For abuse of notation, we also denote by $\Hund_r$ what we call ``$H-$unduloid'': the CMC$-H$ surface obtained gluing together the graph of $u=\Hund_r\circ s$ and its Alexandrov reflection around the slice $\hh\times\{0\}.$ The name ``unduloid'' was first given in \cite{KST} and it is inspired in the homonym surfaces of $\mathbb{R}^3$.\footnote{If we choose $\phi'(0)=-\infty$, we may define the ``$H$-nodoids''; reflecting their graphs we obtain immersed but not embedded rotational surfaces. see \cite{KST}.}
	
		\
	
		To simplify notation, the circle in $\brg(\hhr)$ that corresponds to all geodesic rays in $\hhr$ with slope $m$ in the factor $\rr$ will be denoted by $\brg\hh\times\{m\}$.

The $l_H$ constant defined for $0\le H \le 1/2$, given by
        \begin{equation}\label{lH}l_H:=\dern,\text{ where
}l_{1/2}:=+\infty. \end{equation}
is an important constant that will be used throughout the paper to study the behavior of the asymptotic boundary of constant mean curvature surface, such a constant was defined and studied by G. Citti and C. Senni in \cite{CS}.

	\begin{proposition}\label{p1}
The geodesic boundary of any $\Hund$ is $\sigma=\brg\hh\times\{-l_H\}\cup \brg\hh\times\{l_H\}$.
	\end{proposition}
	\begin{remark}
	 $\sigma$ is H-fillable curve for $0\le H<1/2$  and the case $H=1/2$ the geodesic boundary is $\{p^+,p^-\}$.   
	\end{remark}
	\begin{proof}	
		For $\HM$, $o\in\HH$ and $r>0$  fixed, consider $\Omega := \hh \setminus B_r(o),$ $s: \Omega \to \rr$ the distance to $\partial B_r(o)$ and 
		$$\Sigma^{+}=\{(p,\Hund_{r} (s(p)))\},$$
		where $\Hund_r$ is given by (\ref{eq-hnod}).
		
		Since $\Sigma^+$ is a radial graph, it is enough to analyse the behavior of the derivative $\frac{d}{ds}(\Hund_r(s))$ in the infinity: 
		$$\lim_{s\rightarrow+\infty}\frac{d}{ds}(\Hund_r(s))=\dern=l_H.$$ Thus, for any arc length parameterized geodesic ray $\gamma:[0,+\infty)\to \hh$ with $\gamma(0)=o,$ the geodesic $\beta(t)=(\gamma(t),l_H t)$ in $\hhr$ is asymptotic to the curve  $(\gamma,\Hund_r\circ\gamma))$. Since all geodesic rays starting from a same chosen point $o$ cover all $\brg \hh$, it holds that $\brg(\Sigma^{+})\subset\brg(\HH)\times\{l_H\}$. Furthermore, if $(p,l_H)\in \brg(\hhr)$, and $\gamma$ is such that $[\gamma]=p$ in $\brg\hh$, then $\lim_{t\to+\infty} (\gamma(t), \Hund_r(\gamma(t)))=\lim_{t\to+\infty}(\gamma(t),l_H t)$, which implies that $\brg(\Sigma^{+}) \supset \brg(\hh)\times \{l_H\}.$

		Since the derivative of $\Hhund_r$ at $s=0$ is $+\infty$, making the Alexandrov reflection in the graph $\Sigma^{+}$ we obtain the complete surface $\Sigma=\Hund_r$ and it is immediate that $\brg(\Sigma)=\brg(\HH)\times\{-l_H,l_H\}$. See Figure \ref{fig:imagen2}.

	\end{proof}
	
	\begin{figure}[h]
		\centering
		\includegraphics[width=0.25\linewidth]{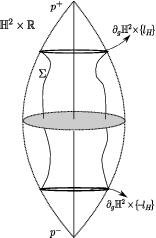}
		\caption{Geodesic boundary of $\Hund_r$}
		\label{fig:imagen2}
	\end{figure}

	Now we consider the the ``$H-$spherical cap'' surface of revolution obtained under the conditions $r=0$ and $\phi'(0)=0$ in the equation \eqref{CMC-edo}: 
	
	\begin{equation}
	\label{eq-def-cap}
			\Hscap(s)=2H\displaystyle\int_0^s \frac{(\cosh(t)-1)}{\sqrt{\sinh^2(t)-4H^2(\cosh(t)-1)^2}}dt. 
	\end{equation}
	
Since $0\le H\le 1/2$, it holds that $\Hscap$ has an entire graph, that we denote also by $\Hscap$. In addition $\lim_{s\to \infty} \Hscap(s)=+\infty$ and  $\lim_{s\to \infty} \Hscap'(s)=2H/\sqrt{1-4H^2}$. The expression explicitly of the $\Hscap(s)$ for $0\le H\le 1/2$ is:

	$$\Hscap(s)=l_Hln \left| \frac{1+\sqrt{\frac{1-4H^2\tanh^2(s/2)}{1-4H^2}}}{1-\sqrt{\frac{1-4H^2\tanh^2(s/2)}{1-4H^2}}}\right|.$$

	\begin{proposition}\label{propgeodboundaryHcap}
	The geodesic boundary of $\Hscap$ is $\brg\hh\times\{l_H\},$ therefore $\brg\hh\times\{l_H\}$ and $\brg\hh\times\{-l_H\}$ are $H-$fillable curves, for $0\le H< 1/2$. 
	\end{proposition}

\begin{proof}

The proof follows reproducing the proof of Proposition \ref{p1} replacing $\Hund$ by $\Hscap$. 
\end{proof}

\subsubsection{Hyperbolic surfaces}\label{sec-hipers}

	In the following we present the surfaces that are invariant under a $1-$parameter family of hyperbolic isometries
	in $\hh$, as defined in \cite{KST}. There is one little difference with the notation, in comparation with that work: while the in that paper the authors consider the distance to a hypercicle and have $[0,+\infty)$ as an uniform domain to all functions of the family, for our applications is more convenient consider always the distance to a geodesic and vary the domain. 
	
	As observed in the beginning of Section \ref{sec:invariant_surfaces}, in this case we consider the distance to a geodesic, giving rise to a family of initial value problems which ODE is
	
	\begin{equation}\label{eq-hipernod}
		\left(\frac{\phi'(s)}{\sqrt{1+\phi'(s)^2}}\right)'+ \left(\frac{\phi'(s)}{\sqrt{1+\phi'(s)^2}}\right) \tanh(s) - 2H=0.\\
	\end{equation}
	
	The initial condition will be chosen after some considerations.
	
	Naming $g=\phi'(1+\phi'^2)^{-1/2}$, the ODE above is rewritten as 
		\begin{equation}\label{eq-hipernod-g}
		g'(s) + \tanh(s) g(s) - 2H=0,
	\end{equation}
and we observe that 
\begin{equation}\label{phieg}
\phi'=\frac{g}{\sqrt{1-g^2)}}, 
\end{equation}
therefore $\phi$ is not defined if $|g|\ge 1$. 
The general solution is given by
	
	$$g(s) = g_C(s) = \frac{2H\sinh(s)+C\sqrt{1-4H^2}}{\cosh(s)},$$
	
where $C$ is some constant to be chosen. 

The reason why we multiply by the non-standard factor $\sqrt{1-4H^2}$ is that this choice has a geometrical meaning (see Remark \ref{remarkMC}, item (iii)) that also simplify computations, as follows: setting $R:=\tanh^{-1}(2H)$, it holds that $\sinh(R)=\dfrac{2H}{\sqrt{1-4H^2}}$ and $\cosh(R)=\dfrac{1}{\sqrt{1-4H^2}}$, therefore $g$ may be rewritten in the following more symmetric way: 

\begin{equation}\label{gC}
g_C(s) = \frac{\sinh(R)\sinh(s) + C}{\cosh(R)\cosh(s)}.
\end{equation}

If easily follows that $|g(s)|<1$ if and only if $-\cosh(R)\cosh(s) < \sinh(R)\sinh(s) + C <\cosh(R)\cosh(s), $ that is,
		$-(\cosh(R)\cosh(s) + \sinh(R)\sinh(s)) <C < \cosh(R)\cosh(s) - \sinh(R)\sinh(s),$ i.e.,
		
		\begin{equation}\label{MCdedution}
			-\cosh(R+s) < C < \cosh(s-R)  
		\end{equation}

		If $|C|<1$ it is clear that  \eqref{MCdedution} is satisfied for all $s\in \mathbb{R}$, hence \eqref{phieg} has solution defined in $\mathbb{R}$.
		
		If $C\ge 1,$ then the first inequality in \eqref{MCdedution} is trivially satisfied and the second one holds if and only if $|s-R| > \cosh^{-1}(C)$, that is, if $s\in (-\infty,R-\cosh^{-1}(C))\cup(R+\cosh^{-1}(C),+\infty)$. Analogously, if $C\le -1$, then the second inequality in \eqref{MCdedution} is always satisfied and the first one holds if and only if $s\in(-\infty,-R-\cosh^{-1}(-C))\cup(\cosh^{-1}(-c)-R,+\infty)$. For our construction it is convenient to consider the cases where $s\to +\infty$, so for each $C$ we consider the maximal interval $(M_C,+\infty)$ where \eqref{phieg} has solution.
		
		Furthermore, it is interesting to investigate what happens with $\phi$ as $s\to M_C^+$. We first claim that $\phi:(M_C,+\infty)$ may be continuously extended to $M_C$ if $|C|>1$. In fact, by L'Hopital's law,
		$$\lim_{s\to M_C^+} \frac{1-g(s)^2}{s-M_C}= -2\lim_{s\to M_C^+}g(s)g'(s).$$ For $C>1$ it holds that this quantity is $\neq 0,$ since (after a straightforward computation) $$g'(M_C) =2H- \frac{2HC-\sqrt{C^2-1}}{C-2H\sqrt{C^2-1}}.$$ The conclusion is that $g(1-g^2)^{-1/2}$ has the same behaviour of $(s-M_C)^{-1/2}$ near $M_C$, what implies that the integral defining $\phi$ converges if $C>1$. The case $C<-1$ is analogous.
		
		For the case $C=1$, first notice that $M_C=R$ and some straightforward computations give us
		$$\frac{g_1(s)}{\sqrt{1-g_1(s)^2}} = \frac{\sinh R\sinh s + 1}{\sinh s - \sinh R},$$ wich implies (after a chain of change of variables) that $\phi_1$ has the explicit expression (up to additive constants)
		$$\phi_1(s) = \sinh(R) s + \cosh(R)\ln\left|\frac{e^s-e^R}{e^s+e^{-R}}\right|;$$ notice that $\phi(s) \to -\infty$ and $\phi_1'(s) \to +\infty$ as $s\to R^+$. 
		
Analogously, for $C=-1$ one has $M_{-1}=-R$ and $$\phi_{-1}(s) = \sinh(R) s + \cosh(R)\ln\left|\frac{e^s+e^R}{e^s-e^{-R}}\right|,$$ also $\lim_{s\to -R^+}\phi(s) = +\infty,$ with $\lim_{s\to -R^+} \phi'(s)=-\infty$.

In all cases, notice that $\lim_{s\to +\infty} g_C(s) = \tanh(R)=2H,$ so $$\lim_{s\to+\infty} \phi_C(s) = l_H.$$

All conclusions above are summarized in the following result:

	\begin{proposition}\label{HiperC}
	For each $C\in \mathbb{R}$ there exists a maximal interval of the form $(M_C,+\infty)$ where $|g|<1$, described below. Moreover, defining $\phi=\phi_C:(M_C,+\infty) \to \rr$ as a solution of $$\phi'=\frac{g_C}{\sqrt{1-g_C^2}},$$ the behavior of $\phi$ is also described as follows.
		\begin{description}
			\item[(i)]If $ C>1$, then $M_C=R+\cosh^{-1}(C)$ and $\phi$ may be continuously extended to the extreme $M_C$ with $\phi'(M_C) = +\infty$;
			\item[(ii)] If $C=1$, then $M_1=R$. Furthermore, $\phi$ has the explicit expression (up to additive constants)
			\begin{equation}\label{C=1}
			\phi(s) = \sinh(R) s + \cosh(R)\ln\left|\frac{e^s-e^R}{e^s+e^{-R}}\right|,
		    \end{equation}
			also $\lim_{s\to R^+}\phi(s) = -\infty,$ with $\lim_{s\to R^+} \phi'(s)=+\infty$;
			\item[(iii)] $-1 < C < 1$, then $M_C=-\infty$, giving rise to what we call a family of ``hipercaps''. It also holds that $\lim_{s\to +\infty} \phi'(s) =l_H.$\footnote{Although $\lim_{s\to+\infty}\phi'(s) =l_H, $ the points at $\brg \hhr$ corresponding to $s=+\infty$ are at height $l_H$ in $\brg\hhr$ since the geodesic ray must be taken going to infinity.}
			\item[(iv)] If $C=-1$, then $M_{-1}=-R$ and $\phi$ has the explicit expression (up to additive constants)
			$$\phi(s) = \sinh(R) s + \cosh(R)\ln\left|\frac{e^s+e^R}{e^s-e^{-R}}\right|,$$ also $\lim_{s\to -R^+}\phi(s) = +\infty,$ with $\lim_{s\to -R^+} \phi'(s)=-\infty$;
			\item[(v)] If $C<-1$, then $M_C=\cosh^{-1}(-C) - R$ and $\phi$ may be continuously extended to $M_C$ with $\phi'(M_C) = -\infty$.
		\end{description}
		In all cases above it holds that $\lim_{s\to +\infty} \phi(s) = +\infty$ with $\lim_{s\to +\infty} \phi'(s)=l_H$.
	\end{proposition}

	\begin{figure}[h]
		\centering
		\includegraphics[width=0.7\linewidth]{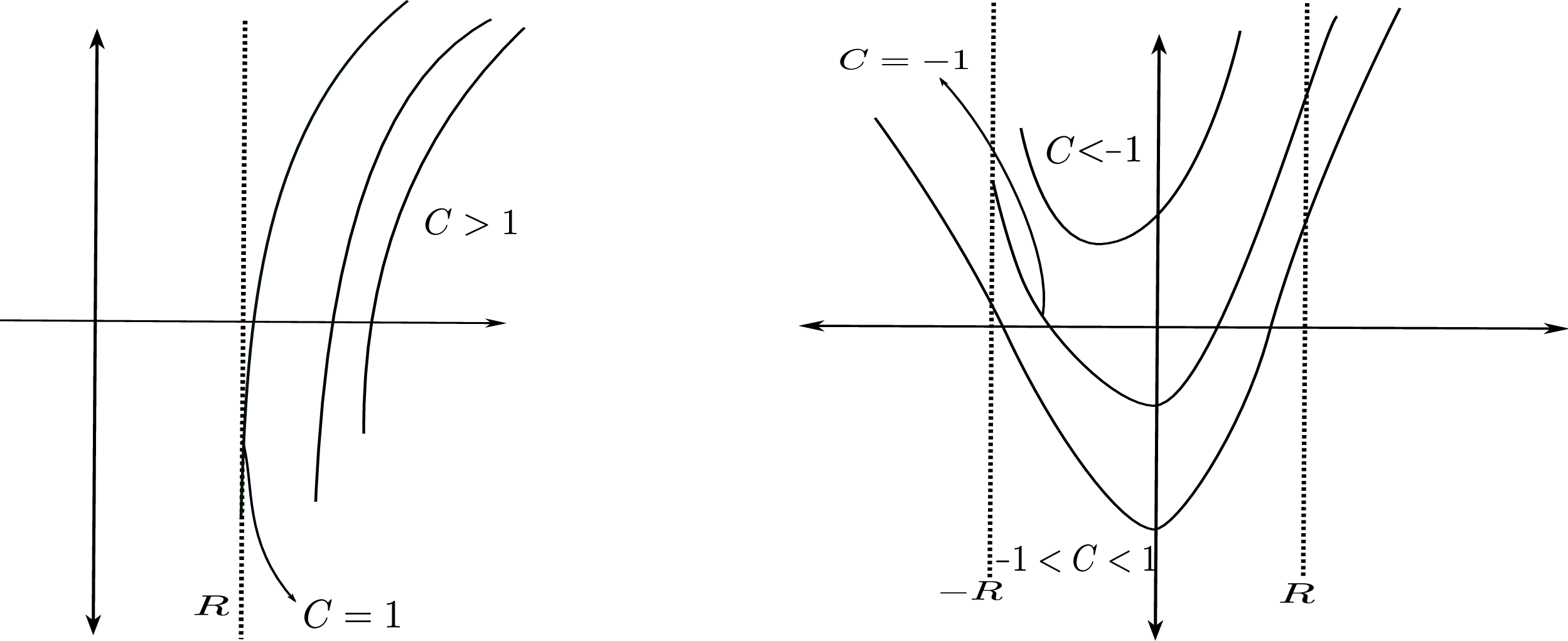}
		\caption{Proposition \ref*{HiperC} }
		\label{fig:hiperc}
	\end{figure}

	\
	
	\begin{remark}\label{remarkMC}
		\begin{enumerate}
			\item The function $\phi$ has a minimum in cases (iii)-(iv), since their values go from $+\infty$ to $+\infty$. This minimum occurs at the point $x_m$ satisfying $g(x_m)=0,$ i.e. $$x_m=\ sinh^{- 1}\left(\frac{C}{\sinh(R)}\right).$$
			\item Since vertical translations are isometries in $\hhr$, we from now on suppose without loss of generality that $\phi(M_C)=0$ in cases (i) and (v), that is,
			$$\phi(s) = \int_{M_C}^s \frac{g(t)}{\sqrt{1-g(t)^2}}dt.$$ Although this simplify some computations, it is not the best way of think about the family of functions $\phi_C$ that we have produced, in this case we should think that the initial condition goes to $-\infty$ as $C\to 1^+$ or to $+\infty$ as $C\to -1^-$.
			\end{enumerate}
			
			\end{remark}
			
			\begin{remark}\label{equidistante}
			   Classical theorems on CMC surfaces of $\mathbb{R}^3$, as done by Jenkins and Serrin \cite{JS} on the 60's give conditions on the geodesic curvature of a bounded domain $\Omega\subset\rr^2$ to guarantee that a CMC graph defined on $\Omega$ take values $\pm \infty$ on portions of the boundary. The same idea occurs in the surfaces corresponding to $C=\pm 1$ as above. The geometric idea is that the surface must be asymptotic to a cylinder over a curve of geodesic curvature $2H$. Precisely, if $R=\tanh^{-1}(2H)$ and if $\beta_R$ is the curve at distance $R>0$ to the geodesic joining $q_1,q_2\in \brg \hh$, then the surface $\beta_R\times\rr$ has constant mean curvature $H=\tanh(R)/2$  and its geodesic boundary is $W(q_1)\cup W(q_2)$.
	\end{remark}

	\
	
	We now return our attention to some hyperbolic surfaces for $0\le H<1/2$. Let $\gamma$ be a fixed geodesic, $s:\hh\to \mathbb{R}$ the distance with sign (as defined in the beginning of this Section) and for each $r\ge R=\tanh^{-1}(2H)$, define
	
	\begin{equation}\label{omegaC}
		\Omega_r:= \{x\in \hh | s(x)>r\};\,\, \beta_r:=\partial \Omega_r = \{x\in \hh | s(x) = r\}.
	\end{equation}
	
	Choosing $C\ge 1$ as the constant such that $M_C=r$, that is, $C=\cosh(r-R)$, by the previous Proposition the function $\Hhund_r:\Omega_r \to \mathbb{R}$ given by $\Hhund_r(x)=\phi_C(s(x))$ is well defined and its graph 
	\begin{equation}\label{Hhundgraph}
		\Sigma_r^+:=\{(x,\Hhund_r(x)) | x\in \Omega_r\}
	\end{equation}
	has constant mean curvature $H$ (with normal vector pointing upwards). Furthermore, if $r>R$ (what implies that $C>1$), then $\Hhund_r$ may be continuously extended to $\beta_r$, having vertical slope in $\beta_r$, which implies that it can be glued with its reflection $\Sigma^-$ around the plane $\hh\times \{0\}$ to produce a immersed CMC surface that we also denote by $\Hhund$. This surface is also embedded and we call it a ``hyperunduloid'', in analogy with the unduloids of the previous Section.\footnote{Analogously, for the cases $C<-1$ (corresponding to another domains), we can construct complete CMC $H$ surfaces with self-intersection and we call them as ``hypernodoids''.  Again the inspiration on these names come from Delaunay's surfaces in Euclidean 3-space, see\cite{KST}.}

	\subsubsection*{Geodesic boundary of $\Hhund_r$}
	
	The rectangles of Proposition 2.1 in \cite{SET} are minimally fillable curves used as barriers in \cite{KLM} , in the case of $H-$fillable curves we build curves which in the particular case of $ H = 0 $ coincide with those rectangles; these curves are also used as barriers in subsequent results.
	
	In the following proposition we show that the geodesic boundary of the $\Hhund$ are like rectangles; the most interesting is that their heights (slopes in $\brg(\hhr)$) depend on the curvature, being zero for $ H = 0 $ and that are approaching to infinite for $ H = 1/2 $.
	
	Given $q_1$ and $q_2$ distinct points at $\brg \hh$ we denote by $\widehat{q_1,q_2}$ one of the arcs in $\brg \hh$ defined by them, that is, the closure of one of the connected components of $\brg \hh \setminus\{q_1,q_2\}.$ When there is no ambiguity, we do not specify explicitly which arc is considered. 
	
	\begin{proposition}
		\label{barreirastall}
		For any $r>R$, the geodesic boundary $\Hhund_r$ is the rectangle $\sigma=\widehat{q_1, q_2} \times \{-l_H,l_H\} \cup (W _ {[-l_H,l_H]} (q_1) \cup W _ {[-l_H,l_H]}(q_2 ))$, where $l_H=\dern$. In particular, $\sigma$ is $H-$fillable with $0< H<1/2$.
	\end{proposition}
	
	\begin{proof}
		
		Let $\gamma$ be the geodesic in $\hh$ with endpoints $q_1$ and $q_2$, $s:\hh\to \mathbb{R}$ the distance to $\gamma$ with sign (as defined in the beginning of this Section) such that $\brg\Omega^+ = \widehat{q_1,q_2}$ and choose  $r>R$, defining $\Omega_r$, $\beta_r$ and $\Sigma_r^+$ as in \eqref{omegaC} and \eqref{Hhundgraph}.
		
		Again for $\HM$  we get 
		$$ \lim_{s\rightarrow + \infty} \Hhund_r'(s)= \dern$$
		so we get $ \widehat{q_1, q_2}\times\{\dern\}\subset\brg\Sigma_r^{+}$.

		In what follows we will prove that $W_{[0,l_H]}(q_1),W _{[0,l_H]}(q_2)\subset\brg\Sigma_r^{+}$.
		
		Let $o\in \gamma$ be a fixed point and consider the cone $K_m$ with slope $m$ given by
		$$K_m=\{(\gamma_1(t), mt)\in\hhr|t\ge 0, \gamma_1\text{ geodesic }, \gamma_1(0)=o, |\gamma_1'|=1\}.$$

		It suffices to prove that for $0\le m \le l_H$ the intersection of $ K_m $ with the graph of $\Hhund_{r}$  converges to points of $W_{[0,l_H]}(q_1), W_{[0,l_H]}(q_2)$ at ``height'' $m$ on the geodesic boundary.
		
		\begin{figure}[h]
			\centering
			\includegraphics[width=0.3\linewidth]{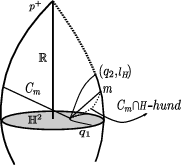}
			\caption{Cone $K_m$}
			\label{fig:cone}
		\end{figure}

		The cone $K_m$ separates $\hhr$ in two connected components (one from above and one from below). We claim that for all $0\le m <l_H$, if $U$ is an open subset (w.r.t. the cone topology) in $\overline{\HH}$ containing $q_1$ (or $q_2$), then there are points on the graph of $\Hhund_{r}$ above and below $K_m$ that project inside $U$.
		
		In fact, let $U$ be an open neighborhood (w.r.t. the cone topology) of $q_1$ and suppose without loss of generality that $q_2 \notin U$. For any point $q\in U\cap\brg\hh\setminus\{q_1\}$ there is a geodesic ray $\eta:[0,+\infty) \to \hh$ such that $[\eta]\equiv q$. Since $q\neq q_1$ and $q \neq q_2$, then $s(\eta(t)) \to +\infty$ as $t\to +\infty$; therefore $\Hhund_r'(s(\eta(t)))\to l_H$, so the graph of $\Hhund_{r}$ go above $K_m$ as $\eta(t)$ goes to $q$. On the other hand, if $\beta(t)$ is a level curve of $s$ in $\Omega_r$ (for instance $s(\beta(t))=1,\forall t$) then $\beta$ is asymptotic to $q_1$(and $q_2$) and $\Hhund_{r}(s(\beta(t)))=\Hhund_{r}(1)$, that is, the height of $\Hhund_{r}$ along the level curve is constant, thus $K_m$ is above $\Hhund(s(\beta(t)))$ for $t$ large. 
		
		Now let $\{U_n\}$ be a sequence of open sets in $\overline{\HH}$ containing $q_1$ and converging to $q_1$, such that the interior of $\partial_g U_n \neq \emptyset$. In each $U_n$ there are $x_n^+,x_n^-$ such that $\Hhund(x_n^+) $ and $ \Hhund(x_n^-) $ are above and below of $K_m$, respectively. Since the graph of $\Hhund_{r}$ is connected, for each $n$ there exists $x_n\in U_n$ such that $(x_n,\Hhund(x_n))\in K_m$, where a first coordinate converges to $q_1$ and second coordinate diverges, in $\overline{\hhr}^g$ it converges to $m$ in the slope sense. It is then proved that  $W_{[0,l_H]}(q_1)\subset\brg\Sigma_r^{+}$. Obviously the same holds for $q_2$.
		
		On the other hand, if $m<0$ or $m>l_H$, then analogously we may observe that for sufficiently small neighborhoods of $q_1$ or $q_2$ all corresponding points of $\Sigma_r^+$ are above or bellow of $K_m$, respectively, concluding the proof that $W(q_i) \cap \brg \Sigma_r^+ = W_{[0,l_H]}(q_i)$.
		
		For obtaining the desired surface, we make Alexandrov reflection of $\Sigma_r^+$, obtaining $\Hhund$, see Figure \ref{fig:imagen4}.

	\end{proof}
	\begin{figure}[h]
		\centering
		\includegraphics[width=0.22\linewidth]{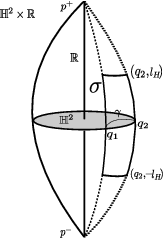}
		\caption{Proposition \ref{barreirastall}}
		\label{fig:imagen4}
	\end{figure}
	
	\begin{proposition}
		\label{ExemplosNovos}
		Let $\sigma_1=(\widehat{q_1,q_2}\times\{l_H\})\cup W _{[-\infty,l_H]}(q_1)\cup W _{[-\infty,l_H]}(q_2)$ and $\sigma_2=(\widehat{q_1,q_2}\times\{l_H\})\cup W _{[0,+\infty]}(q_1) \cup W _{[0,+\infty]}(q_2)$, where $l_H=\dern$. Then $\sigma_1$ and $\sigma_2$ are geodesic boundary of a properly embedded surface with constant mean curvature $H$, with $0\leq H<1/2$.
	\end{proposition}
	
	\begin{proof}
		
		The proof follows reproducing the proof of Proposition \ref{barreirastall} replacing $C>1$ by $C=1$ and $C=-1$, respectively.
		
	\end{proof}
	
	\begin{figure}
		\centering
		\includegraphics[width=0.4\linewidth]{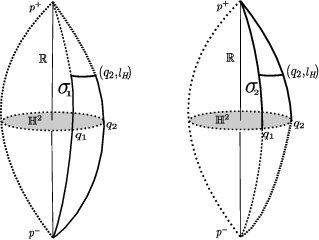}
		\caption{Proposition \ref{ExemplosNovos}}
		\label{fig:exnovos}
	\end{figure}

	\begin{remark}
	  Other invariant surfaces are the invariant surfaces under parabolic isometries as mentioned in section \ref{sec:invariant_surfaces}, it is possible to make a similar construction to the invariant surfaces under parabolic and hyperbolic isometries, but in this case the geodesic boundary of the surfaces is not an H-fillable curve, where the geodesic boundary are two circles of height $l_H$ and $-l_H$ respectively and an interval of a Weyl chamber.
	\end{remark}

	\bigskip
	
	\bigskip

	\section{$H-$fillable curves classification}

	The previous results allow us to have intuition that the geodesic boundary of the surface in $\hhr$ with constant mean curvature  $ H \in [0,1 / 2] $ is determined by $H$, increasing from the equator when $H=0$ to the poles (as $H\to 1/2$). The following results shows precisely which curvature $H$ gives the parameters to classify the properly embedded CMC$-H$ surfaces in $\hhr$. 
	
	\begin{theorem}\label{teo1}
	    Let $\Sigma$ be a properly embedded surface with constant mean curvature $0\le H< 1/2$ such that $\brg \Sigma$ intersects $W(q)$ transversally at $(q,m)$. Then $m=\pm l_H.$
	\end{theorem}

	\begin{proof}
		The proof follows the some ideas of the proof of Theorem 5.1 in \cite{KLM}. Let $m>0$ be such that $\brg \Sigma$ is transversal to $W(q)$ at $(q,m)$ and by contradiction we first suppose that  $m>l_H$.
		
		Then there exists $l$ such that $l_H<l<m$. By transversality there exists an open neighbourhood $\mathcal{U}$ of $(q,l)$ such that $\mathcal{U}\cap\Sigma = \emptyset$. Also let us denote by $K_l^+$ the connected component of $\hhr$ that is above the cone $K_l$ with slope $l$, centred at a previously selected point $o\in\hh$. Denote by $\Sigma'$ be the intersection $\Sigma \cap K_l^+$, that is, the portion of $\Sigma$ that is above the cone.

		Let $c\in \brg \hh$ be an arc containing $q$ as interior point and such that $c\times\{l\}\subset \mathcal{U}$. Let $\gamma\in \hh$ be a geodesic and $\Omega^+\subset \hh$ be the semi space determined by $\gamma$ such that $\brg\Omega^+=c$. Let $s:\Omega^+\to [0,+\infty)$ be the distance to $\gamma$ as in Section \ref{sec:invariant_surfaces}. Choose any $r>R$ in Proposition \ref{HiperC} and consider the surface $\Hhund_r$. By the choice of $\mathcal{U}$ we know that $\brg\Sigma'\cap\brg\Hhund_r = \emptyset$. 
		
		For all $t\in \mathbb{R}$, let $$S_t:=\{(x,\Hhund_r(s(x))+t),\,x\in \Omega_r\},$$ that is, the vertical translation of $\Hhund$ by $t$. Also by the choice of $\mathcal{U}$ there exists $t_0\in \mathbb{R}$ such that $S_{t_0}\cap \Sigma'=\emptyset$. On the other hand, notice that for sufficiently large $t$ it holds that $S_t\cap \Sigma'\neq\emptyset$. It follows that $\exists t^*>t_0$ such that $S_{t^*}\cap \Sigma'\neq \emptyset$ but $S_t \cap \Sigma'=\emptyset,\; \forall t<t^*$. This must produce a tangent point. It does not occur at the interior boundary of $\Sigma'$ by the choice of $\mathcal{U}$, contradicting the Tangency Principle.

		
		\begin{figure}[h]
			\centering
			\includegraphics[width=0.25\linewidth]{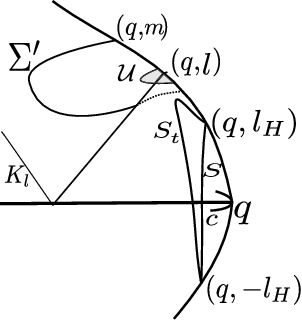}
			\caption{$l_H<m$ of theorem 4.1}
			\label{fig:imagen5}
		\end{figure}
		
		\

	In the other case, let us suppose by contradiction that  $0< m < l_H$. Then there is $l$ such that $m<l<l_H$, thus  
	there exists an open neighborhood $\mathcal{U}$ of $(q,l)$ such that $\mathcal{U}\cap \Sigma = \emptyset$.
	Let $\Sigma'$ be the portion of $\Sigma$ that contains $(q,m)$ in its geodesic boundary and such that $\Sigma'$ is below $K_l$ ( $\Sigma'=K_l^-\cap\Sigma$). Define $\mathcal{V}\subset \hh$ as the projection of $\mathcal{U}$ over $\hh$, that is an open neighborhood of $q\in \brg \hh$. Choose any point $a\in \mathcal{U}$ and consider $S_0=\Hscap_a$ a spherical cap centered at $a$. As in the previous cases, the family of vertical translations $\{S_t\}_t\in \mathbb{R}$ of $S_0$ must have a last contact point with $\Sigma'$, getting a contradiction.

    Finally, the case when $\brg\Sigma$ intersects transverselly $W(q)$ in $(q,0)$ ($m=0$), in this case $\Sigma$, by Lemma \ref{LemaOrienta}, can be oriented either upward or downward in a neighborhood of $(q,0)$, if $\Sigma$ is oriented upward we consider barriers $\Hscap$ and use a similar process to $l>0$. Also for $\Sigma$ oriented downward, taking the barrier $-\Hscap$.	\end{proof}

		\begin{figure}[h]
			\centering
			\includegraphics[width=0.3\linewidth]{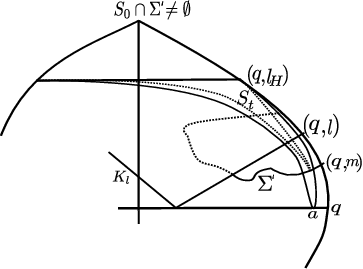}
			\caption{$0<m<l_H$ of theorem 4.1}
			\label{fig:imagen06}
		\end{figure}

	In this text the geodesic boundary studied are $H-$fillable curves and according to the theorem \ref{teo1} the $H-$fillable curves run through by a fixed Weyl Chamber and can only change to another Weyl Chamber over the sets $\brg\HH\times\{-l_H,l_H\}$ or at the poles. In other cases we could have surfaces with geodesic boundary more complex, for example the surface presented by A. Folha H. and Rosenberg in \cite{FR}, where the geodesic boundary is the whole $\brg\HH\times\rr$. Also $\cite{KLM}$ in the Theorem 5.7 presents an example of a minimal surface with geodesic boundary being a strip of $\brg(\hhr)$. In our case, the following Corollary presents all possible $H-$fillable curves in $\brg(\hhr)$.

	\begin{corollary}\label{cor1}

		For given $0<H<1/2$, any $H-$fillable curve in $\brg(\hhr)$ is one of the following:
		\begin{enumerate}[(i)]  
			
			\item The curves $\brg\HH\times\{l_H\}$ or $\brg\HH\times\{-l_H\}$.
			
			\item The union of arcs with height
			$ l_H $ (or $ -l_H $) with Weyls Chambers intervals from $ l_H $ (or $-l_H$) to $ p^+ $ (or $ p^-$).
			
			\item The union of arcs with height $l_H$ and $-l_H$ with Weyl Chambers intervals from $ -l_H $ to $ l_H $.
			
			\item The union of arcs with height $ l_H $ and $ -l_H $ with Weyl Chambers intervals from $-l_H$ to $l_H$, $l_H$(or $-l_H$) to $p^+$ (or $p^-$), the arcs can be reduced by one point.
			
			\item The union of arcs with height $ l_H $ and $ -l_H $ with Weyl Chambers intervals from $p^-$ to $-l_H$, $-l_H$ to $l_H$, $l_H$ to $p^+$, the arcs can be reduced by one point.
			
			\item Finally the union of curves type i, ii, iii, iv when possible.
		\end{enumerate}
		
	\end{corollary}
	
	\begin{proof}
		
		Note that from Theorem \ref{teo1}, if $\sigma$ is $H-$fillable we get that the intersection of $\sigma $ with the Weyl Chambers $W(q)$ is or empty or some closed set $W_{[l_H,+\infty]}(q)$, $W_{[-\infty,-l_H]}(q) $, $W_{[-l_H,l_H]}(q)$,$\{(q,-l_H)\}$ and $\{(q,l_H)\}$. This means that $\sigma$ can only go through Weyl Chambers and change from one Weyl Chamber to another in the sets $\brg\HH\times\{l_H\}$, $\brg\HH\times\{-l_H\}$, $\{p^+\}$ or $\{p^-\}$. Under these conditions and doing all the possibilities we get each of the curves to corollary. See Figure \ref{fig:imagen7}.
	\end{proof}
	\begin{figure}[h]
		\centering
		\includegraphics[width=0.42\linewidth]{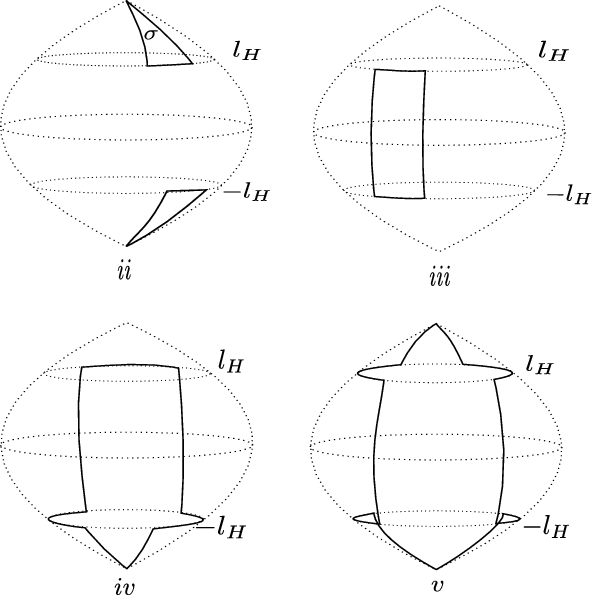}
		\caption{Corollary \ref{cor1}}
		\label{fig:imagen7}
	\end{figure}
	
	\begin{example}
		As  examples of classification curves we have the proposition \ref{p1} curves which are the type (i)  also the proposition \ref{barreirastall} is an examples of type (iii). In proposition \ref{ExemplosNovos} we have we have a curve that is an example of type (iv). Finally the geodesic boundary of the surface $\beta_{R_H}\times\rr$  are curves of type (v).
		
	\end{example}

	\begin{corollary}\label{cor2}
		There is no  properly embedded surface $\Sigma$ into $\hhr$ with constant normal vector field continuous up to $\brg\hhr$ and mean curvature $0<H\leqslant1/2$ such that $\brg\Sigma\subseteq\brg\HH$.
		
	\end{corollary}
	
	\bigskip
	\bigskip
	\bigskip

	\section{Product boundary $\mathcal{B}_H\times\rr$}\label{sec:B_H}

	For the final result we have two important references, the first is the solution of Plateau's asymptotic problem for minimal surface equations, made by B. Nelli and H. Rosenberg in  \cite{NR} 
	
	\begin{theorem}[Theorem 4 of \cite{NR}]\label{theonellirosenberg}
		Let $\Gamma$ be a continuous Jordan curve in $\brg(\HH)\times\rr$ that is a vertical graph. Then there exists a minimal vertical graph on $\HH$ having $\Gamma$ as asymptotic boundary. The graph is unique.
	\end{theorem}
	
	The result above can also be understood as the existence of minimal fillable curves on the product boundary. 
	
	The second important reference is from Jiang, Wang, and Zhu in \cite{JWZ}, where it is made a solution of the  Liouville type problem in $\rr_+^n,\,n\geq2$ for minimal surfaces.
	
	\begin{theorem}[Theorem 1.1 of \cite{JWZ}]
		Let $n\geq2$ be an integer and $u\in C^2 (\rr_+^n)\cap C(\overline{\rr_+^n})$ be a solution of
		
		\begin{equation}
			\begin{array}{ccc}
				
				\dive\left(\frac{\nabla u}{\sqrt{1+|\nabla u|^2}}\right)&=&0\,\,\,in\,\,\,\rr_+^n\\
				u                                                       &=&l\,\,\,on\,\,\,\partial\rr_+^n
			\end{array}
		\end{equation}
		where $l:\rr^n\rightarrow\rr$ is an affine function. Assume that $u:\overline{\rr_+^n}\rightarrow\rr$ has at most a linear growth, which means there exists a constant $K>0$ such that
		
		$$|u(x)|\leq K(|x| + 1)\,\,\,for\,\ any\,\ x\in\overline{\rr_+^n}.$$
		
		Then $u$ is an affine function.
	\end{theorem}
	
	In this case the solution on the boundary is an affine function.

	In this article we extend the result of \cite{NR} to CMC$-H$ surfaces in $\hhr$ and we would not have affine solutions on the boundary as in \cite{JWZ}  but continuous functions going to infinity with slope $l_H$.
	
	\bigskip
	\bigskip

	A way of understating these surfaces is the following: the geodesic boundary for minimal graphs are the equator $\brg\HH\times\{0\}$, then to have a better notion of their behavior we may consider the product boundary $\brg\HH\times\rr$. In this product, as mentioned above, graphs of continuous functions defined over $\brg \hh$ are the geodesic boundary of minimal graphs over $\hh$. 
	
	In the other hand, the case of  CMC$-H$ surfaces with $0<H<1/2$, the geodesic boundary explodes as it approaches $\brg\HH$ (Corollary \ref{cor2}), then the product boundary $(\brg \hh) \times \rr$ is useless. An alternative to this problem is to consider another kind of product boundary that occurs at ``height'' $l_H$, where we expect that many different graphs of CMC-$H$ converge to. In order to distinguish them we consider the product defined in the following steps.

	\begin{definition}
		Let $o$ be a fixed origin in $\hh$ and $\Hscap=\Hscap_o$ the $H-$spherical cap centered at $o$. We say that $c(t)$ is asymptotic to $\Hscap$ at height $C$ along the $\gamma$ if $c(t)-\Hscap(\gamma(t))\to C$ for $t\to+\infty$. 
		
	\end{definition}
	
	S. Cartier and L. Hauswirth introduced in \cite{CH} the definition of asymptotic distance from a surface to a hyperboloid $S$ fixed with  geodesic boundary $p^+$ or $p^-$, a definition that in our case would be an asymptotic height for $H=1/2$. 
	\bigskip
	

	\begin{definition}
		Given $0\le H<1/2$ and $\Hscap$ be fixed at a origin $o\in\hh$, denote by $\mathcal{B}_H=\brg\Hscap$ and define the product boundary over $\mathcal{B_H}$ by $\mathcal{B}_H\times\rr$, where the factor $\rr$ represents the asymptotic heights with respect to $\Hscap$.
	\end{definition}

	The previous definition provides more information on the asymptotic behavior of surfaces that are asymptotic at $\Hscap$. Remember that $\overline{\Hscap}^g=\Hscap\cup\mathcal{B}_H\subset\overline{\hhr}^g$ is compact subset, so it is a compact topological space with the subspace topology. Thus, $\overline{\Hscap}\times\overline{\rr}$ is a compact topological space with the product topology and whose product boundary is $\mathcal{B}_H\times\overline{\rr}$. Furthermore, if a $\Sigma$ surface is asymptotic to $\Hscap$, then we will denote its boundary in $\mathcal{B}_H\times\overline{\rr}$ by $\partial_{\times_H}\Sigma$. In the particular case $H=0$, we have $\partial_{\times_H}\Sigma\equiv\partial_{\times}\Sigma$, see Figure \ref{fig:shxr}.

	\begin{figure}[h]
		\centering
		\includegraphics[width=0.33\linewidth]{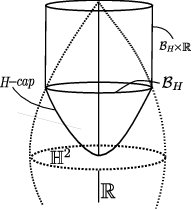}
		\caption{$\brg\Hscap\times\rr$}
		\label{fig:shxr}
	\end{figure}

	A supersolution for $Q_H$(operator defined in \eqref{cmc-equation}) in $\hh$ is function $w^+\in C^{0}(\hh)$ such that, given any bounded domain $\varLambda\subset \hh$, if $u\in C^0(\bar{\varLambda})\cap C^2(\varLambda)$ satisfies $Q_H(u)=0$ in $\varLambda$ and $u|_{ \partial\varLambda}\le w^+|_{\partial\varLambda}$, then $u\le w^+$ in $\varLambda$. Subsolution is defined analogously, replacing ``$\ge$'' by ``$\le$'' on the inequalities above.

	\begin{definition}
		Let $\Hscap$ be fixed at a origin $o\in\HH$. Given $(q,l_H)\in \mathcal{B}_H$ and $\Omega\subset\HH$ such that $q$ is an interior point of
		$\brg\Omega$. A lower barrier (resp. upper) for $Q_H$ relative to $(q,l_H)$ and $\Omega $ with
		asymptotic height $C$ is a function $w^-\in C^0(\HH)$ (resp. $w^+\in C^0(\HH)$) such that
		\begin{enumerate}[(i)]
			\item $w^-$ is subsolution (resp. $w^+$ is a supersolution) of $Q_H=0$.
			\item $w^- -\Hscap\le0$ (resp. $w^+-\Hscap\ge0$) and $\lim_{x\to q}(w^-(x)-\Hscap(x))=0$ (resp.$\lim_{x\to q}(w^+(x)-\Hscap(x))=0$).
			\item $(w^- -\Hscap)_{\HH\backslash\Omega}\le-C$ (resp. $(w^+-\Hscap)_{\HH\backslash\Omega}\ge C$).
		\end{enumerate}
		Here the limit $x\to q$, $x\in\HH$, $q\in\brg\HH$, is in terms of cone topology.
	\end{definition}


\begin{proposition}\label{barinf}
Let $q\in \brg\hh$ be a chosen point and $\Omega$ some neighborhood of $q$ w.r.t. the cone topology, and $C>0$ be a positive constant. There is a lower barrier for $Q_H$ relative to $(q,l_H)$ and $\Omega$ with asymptotic height $-C$.   
\end{proposition}

\begin{proof}
Denote by $\gamma:[0,+\infty)\to \hh$ the geodesic ray such that $\gamma(0)=o$ and $\gamma(+\infty)=q$, that is, $[\gamma]=q$. Furthermore, let $\beta:\mathbb{R}\to \hh$ be a complete geodesic such that $\beta(\mathbb{R})\subset \Omega$ and such that $\beta$ is orthogonal to $\gamma$ on their intersection point $x_0$, that is a point in $\Omega$. Also, denote by $D$ the distance from $o$ to $\beta$, that coincides with the distance from $o$ to $x_0$. 

Let $\beta_R$ be an equidistant curve to $\beta$, at distance $R=\tanh^{-1}(2H)$ (see Remark \ref{equidistante}). Let $\Omega'$ be the connected component of $\hh \setminus \beta_R$, i.e., the connected component that does not contain $\beta$. In particular, $\Omega'\subset \Omega$. Finally, let $v^-:\Omega' \to \mathbb{R}$ a vertical translation of the function given by \eqref{C=1}, that is,

		$$v^-(x)=l_H s + \frac{1}{\sqrt{1-4H^2}}\ln\left|\frac{e^s-e^R}{e^s+e^{-R}}\right|-c,$$ where $s:\Omega'\to \mathbb{R}$ denotes the distance to $\beta$ and $c$ is going to be determined in the next steps. 
		
		
		In Section \ref{sec-hipers} it was proved that $v^-$ has CMC graph that is asymptotic to $\Hscap$ in the interior of $\Omega'$.

		First we will prove that if $c=c(D)$ is correctly chosen, then the asymptotic height of $v^-$ w.r.t. $\Hscap$ is zero at $q$. 
		
		Let us starting rewriting $\Hscap$ as bellow, where $s$ is the distance function to $o$:

		$$\Hscap(x)=l_Hln \left| \frac{1+\sqrt{\frac{1-4H^2\tanh^2(t/2)}{1-4H^2}}}{1-\sqrt{\frac{1-4H^2\tanh^2(t/2)}{1-4H^2}}}\right|, $$ where $t(x)$ is the distance from $o$ to $x$.

		If $x$ is a point over $\gamma$ that also belongs to $\Omega'$, by the orthogonality between $\gamma$ and $\beta$ and by definition of $D$, it holds that $t(x)=s(x)+D$. Therefore 
		
		\begin{equation}
		    \label{v-Hcap}
		    v^-(x)-\Hscap(x) = l_H s + \frac{1}{\sqrt{1-4H^2}} \ln \left|\frac{e^s - e^R}{e^s+e^{-R}}\right|-c-l_H\ln \left|\frac{1+\sqrt{\frac{1-4H^2\tanh(\frac{s+D}{2})}{1-4H^2}}}{1-\sqrt{\frac{1-4H^2\tanh(\frac{s+D}{2})}{1-4H^2}}}\right|
		\end{equation}
		
		Since the second term in \eqref{v-Hcap} goes to zero and $\tanh(\frac{s+D}{2})$ goes to $1$ as $s$ goes to infinity, it holds that the limit of $v^-(x)-\Hscap(x)$ as $x\to q$ along $\gamma$ is given by the following expression:

		$$l_H\lim_{s\to \infty}\left[ s-\ln \left|\frac{1}{2}-\frac{1}{2} \sqrt{\frac{1-\tanh^2(\frac{s+D}{2})}{1-4H^2}}\right| \right],$$
		
		that is, 
		
		$$\lim_{x\to q, x\text{ along }\gamma} (v^-(x) - \Hscap(x)) = l_H \lim_{s\to+\infty} ln \left| \frac{e^s \left(1-\sqrt{\frac{1-4H^2\tanh^2(\frac{s+D}{2})}{1-4H^2}}\right)}{2}\right|-c.$$
		
		We may rewrite the expression above and use L'Hopital's law to compute the limit inside the logarithm as bellow:
		
		$$\lim_{s\to \infty} \left|\frac{e^s \left(1-\sqrt{\frac{1-4H^2\tanh^2(\frac{s+D}{2})}{1-4H^2}}\right)}{2}\right|= \lim_{s\to+\infty} \left|\frac{1-\sqrt{\frac{1-4H^2\tanh^2(\frac{s+D}{2})}{1-4H^2}}}{2e^{-s}}\right|$$
		
		$$=\lim_{s\to +\infty} \left(\frac{1-4H^2\tanh^2(\frac{s+d}{2})}{1-4H^2}\right)^{-1/2}l_H^2\frac{e^s\tanh(\frac{s+D}{2})}{4\cosh^2(\frac{s+D}{2})},$$
		
		and since $\tanh(\frac{s+D}{2}) \to 1$ and $$\lim_{s\to +\infty} \frac{e^s}{\cosh^2(\frac{s+D}{2})}=\lim_{s\to +\infty} \frac{4e^s}{e^{s+D}+e^{-(s+D)}}=4e^{-D},$$ we finally obtain
		
		$$\lim_{x\to q, x\text{ along } \gamma} v^-(x)-\Hscap(x) = l_H (\ln(2 l_H^2)-D)-c,$$
		
		giving us 
		
		$$c(D) = l_H (\ln(l_H^2)-D).$$
		
		Since the functions $v^-$ and $\Hscap$ are continuous w.r.t. cone topology, the computed limit coincides with the limit for $x\to q$, not necessarily along $\gamma$. Therefore, the asymptotic height of $v^-$ w.r.t. $\Hscap$ is $0$ in $q$. Furthermore, continuity of $v^-$ and $\Hscap$ in the interior of $\brg \Omega'$ also implies that all the asymptotic heights of $v^-$ w.r.t. $\Hscap$ are well defined in a neighborhood of $q$.
		
		Next we need to prove that in fact $q$ is a maximum point of $v^{-}-\Hscap$. For this propose, we analyse what happens $\Hscap$ in the level sets of $v^-,$ which are the equidistant curves $\beta_s$, $s> R,$ to $\beta$. To simplify notation, denote by $x_s$ the intersection between $\beta_s$ and $\gamma$.

	\begin{figure}[h]
	\centering
	\includegraphics[width=0.3\linewidth]{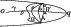}
	\caption{}
	\label{fig:sub}
\end{figure}
		
		If $s>R$ and $x$ is a point in $\beta_s$, the distance from $x$ to $o$ is greater then or equal to the distance from $x_s$ to $o$, with equality if and only if $x=x_s$. Since $\Hscap$ is a strictly increasing function, we have $\Hscap(x) \ge \Hscap(x_s),$ with equality  if and only if $x=x_s$. Furthermore, since $x\in \beta_s,$ it holds that $v^-(x) = v^-(x_s)$. These facts together imply that
		
		$$v^-(x) - \Hscap(x) \le v^-(x_s) - \Hscap(x_s),$$ with equality  if and only if $x=x_s$, what gives us that the maximum of $v^- - \Hscap$ in $\brg \Omega'$ must occur at $q$.  
		
		To finish, since $\lim_{x\in \Omega',x\to \beta_R}v^{-} (x)=-\infty$ and the maximum of $v^-$ at each level set of $\Hscap$ is always attained along $\gamma$, we obtain that given any $C>0$, there exists $\widetilde{\Omega}$ open subset of $\Omega'$ such that 

$$v^- \ge (\Hscap-2C) \text{ in }\widetilde{\Omega}.$$

Define $w^-:\HH\rightarrow\rr$ by

$$w^- = \left \{ \begin{matrix}\Hscap-C &\mbox{if }& x\in\HH\backslash  \widetilde{\Omega}\\ 
	v^- &\mbox{if }& x\in \widetilde{\Omega}.  \end{matrix}\right. $$

Therefore $w^-$ is a lower barrier for $Q_H$ relative to $(q,l_H)$, $\Omega$ and constant $-C$. 

\end{proof}

\begin{figure}[h]
	\centering
	\includegraphics[width=0.25\linewidth]{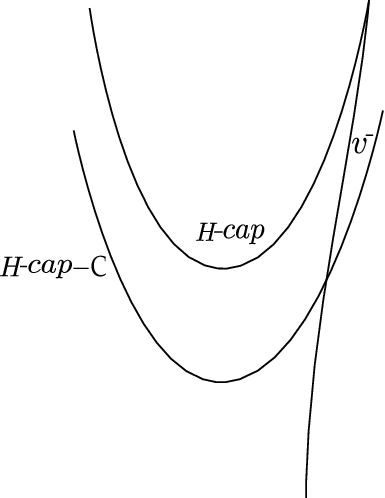}
	\caption{Lower barrier}
	\label{fig:sub}
\end{figure}

\begin{proposition}\label{barsup}
Let $q\in \brg \hh$, $\Omega\subset \overline{\hh}$ open neighborhood of $q$ and $C>0$ be given. Then there exists an upper barrier at $q$ w.r.t. $\Omega$ and $C.$
\end{proposition}

\begin{proof}
Let $\gamma:[0,+\infty)\to \hh$ be the geodesic ray such that $\gamma(0)=o$ and $\gamma(+\infty)=q$. Chose $\beta:\rr\to \hh$ a geodesic such that $\beta(\rr)\subset \Omega$ and $\beta$ is orthogonal to $\gamma$. Without loss of generality, we may assume that $\Omega$ is the connected component of $\hh \setminus \beta(\mathbb{R})$ that contains $q$ on its asymptotic boundary.

We start presenting a supersolution $w:\overline{\Omega}\to \rr$ such that $(w-\Hscap_o)\ge 0$,  $(w-\Hscap_o)|_{\beta} \ge 2C$ and $\lim_{x\to q} (w-\Hscap_o)(x) = 0$.  

{\bf Claim 1.} If $p^*$ is the reflection of $o$ around $\beta$, then $\Hscap_{p^*}(x)=\Hscap_o(x)$ for all $x\in \beta$ and $\Hscap_{p^*}(x)\ge \Hscap_o(x)$ in $\hh\setminus \Omega$.  This is a consequence of hyperbolic trigonometry: since $p^*$ is the reflection of $o$ w.r.t. $\beta$, then for all $x\in \beta$ it holds that $d(x,p^*)=d(x,o)$; also for all $x\in \hh\setminus \Omega$ it holds that $d(p^*,x)>d(o,x).$

{\bf Claim 2.} If $p\in \gamma$ is such that $d(p,o)>d(p^*,o)$, then $(\Hscap_p-\Hscap_o)>0$ along $\beta$. This also is a consequence of hyperbolic trigonometry. Denote by $x^*$ the intersection between $\gamma$ and $\beta$. Therefore for any $x\in \beta$, both $\Delta x^*p^*x$ and $\Delta x^* p x$ are triangles that are rectangle in $x^*$; since $d(x^*,p)>d(x^*,p^*)$ and both triangles share the side joining $x^*$ to $x$, it holds that $d(x,p)>d(x,p^*)=d(x,o)$. 

{\bf Claim 3.} There exists $p \in \gamma$ sufficiently away from $p^*$ such that $$\lim_{x\to q} (\Hscap_p-\Hscap_o)(x) = -2C.$$ 
Indeed: if $p=\gamma(D)$, then a straightforward computation analogous to the ones done in the previous proposition implies that $$\lim_{x\to q}(\Hscap_p-\Hscap_o)(x) = - l_H D.$$ Chose $D= 2C/l_H.$

{\bf Claim 4.} Let $p=\gamma(D)$ as in Claim 3. Then the function $$v^+:=\Hscap_p+2C$$ is the desired barrier. Indeed: by Claims 1 to 3, we have that $(v^+-\Hscap_o)\ge 2C$ in $\hh\setminus \Omega$ and $\lim_{x\to q}(v^+-\Hscap_o)(x) = 0$. We now must prove that $(v^+-\Hscap_o)\ge 0$ in $\Omega$. For this propose, we will prove that for any circle $\mathcal{C}$ centered at $o$, the minimum of $(v^+-\Hscap_o)|_{\mathcal{C}}$ occurs at $\mathcal{C}\cap \gamma$.

If $\{x\}=\mathcal{C}\cap\gamma$, then the distance from $x$ to $p$ is the shortest distance from $\mathcal{C}$ to $p$, since $\gamma$ is orthogonal to $\mathcal{C}$. As $\Hscap_p(s)$ is increasing, we have, for any $y\in \mathcal{C}$, $ \Hscap_p(x)\le\Hscap_p(y)$, therefore $\Hscap_p(x)-\Hscap_o(x)\le \Hscap_p(y)-\Hscap_o(y)\forall y\in\mathcal{C}$.

This proves that $q$ is the minimum of $v^+-\Hscap_o$.

To finish, let us define $w^+:\HH\rightarrow\rr$ by
	
\begin{equation}\label{BarrSup}
		w^+:=\min\{v^+,C\}
\end{equation}
	
	If follows that $w^+$ is a supersolution of $Q_H=0$ in $\hh$ that is is an upper barrier at $(q,l_H)$ w.r.t. $\Omega$ and $C$.

\end{proof}

\begin{figure}[h]
	\centering
	\includegraphics[width=0.4\linewidth]{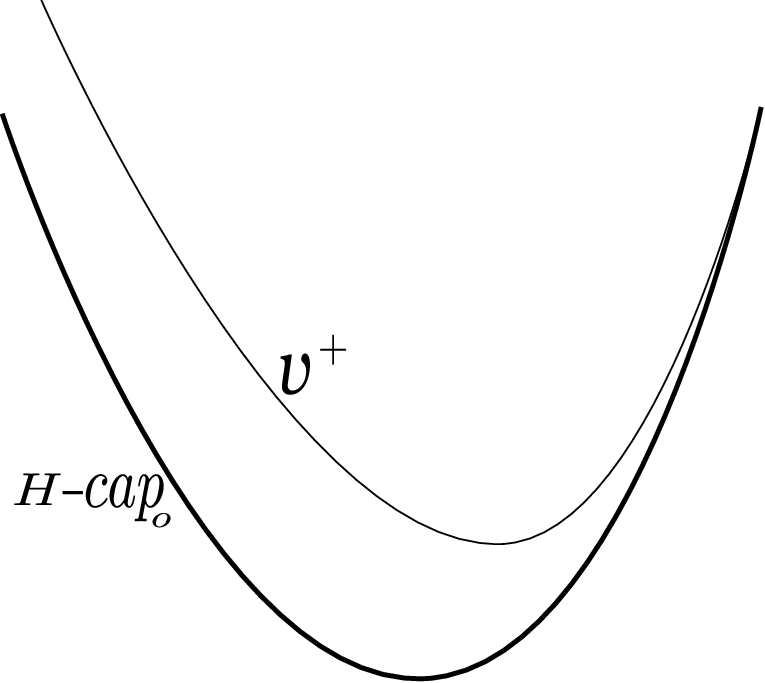}
	\caption{Upper barrier}
	\label{fig:sup}
\end{figure}



Finally, we present the last theorem, that generalises Theorem
\ref{theonellirosenberg} for any $0\le H<1/2$:

\begin{theorem}\label{S_H}
        Let $0\le H<1/2 $ and $\phi\in C^0(\mathcal{B}_H)$. There is a
unique $u\in C^{\infty}(\HH)$ such that $Q_H(u)=0$ and $u$ continuously
extends to $\mathcal{B}_H$ and its  extension $\overline{u}$ satisfies
$\overline{u}|_{\mathcal{B}_H}=\phi$, in the sense of asymptotic heights.
In other words, $$\lim_{x\to p} (u-\Hscap)(x)=\phi(p)$$ for all $p\in
\brg\hh$.
\end{theorem}   

\begin {proof}
We use the Perron method for CMC graph equation, describing lower and
upper barriers at infinity. 

First define
$$S^{\phi}=\{v\in C^0(\bar{\HH})|v\,\text {is a supersolution
of}\,Q_H\,\text{such that}\,(v-\Hscap_o)|_{\mathcal{B}_H}>\phi\}.$$

The set $S^\phi$ is not empty because $w^+ + \max \phi,$ where $w^+$ is
given in the Proposition \ref{barsup}, is a supersolution. It follows
from Perron's Method that the function $u:\hh\to\rr$ given by
$$u(x):=\inf_{v\in S^\phi}v(x),\,x\in{\HH}$$
 is a solution of the CMC graph equation, that is, $u\in
C^{\infty}(\HH)$ with $Q_H(u)=0$. To get our result, it then suffices
to prove that $u$ extends continuously to $\mathcal{B}_H$ and its
extension $\bar{u}$ satisfies $\bar{u}|_{\mathcal{B}_H}=\phi$.

Since asymptotic heights are obtained subtracting $\Hscap_o$,
Comparison Principle is also valid in our context. Precisely, if $v,w$
are global continuous sub and supersolutions for $Q_H$, respectively,
such that the asymptotic heights of $v$ w.r.t. $\Hscap_o$ are less then
or equal to the asymptotic heights of $w$, then $v\le w$ in $\hh$.

If $q\in \brg \hh$ and $\varepsilon>0$, the continuity of $\phi$
implies that there exists an open neighborhood $W$ of $q$ in $\brg \hh$
such that $|\phi(y)-\phi(q)|<\varepsilon/2$ at $W$. Let $\Omega\subset
\hh$ be an open neighborhood of $q$ such that $\brg \Omega$ is
contained in the interior of $W$. Consider $w^+$ an upper barrier at
$q$ with respect to $\Omega$ and with asymptotic height $C=\max |\phi|$
and define $u^+:\hh \to \mathbb{R}$ by $$u^+(x)=\phi(q) + w^+(x) +
\varepsilon.$$
Then $u^+\in S^{\phi}$ and therefore $u\le u^+$, which implies that
$\limsup_{x\to q}(u-\Hscap_o)(x)\le \phi(q)+\varepsilon.$

Furthermore, defining $u^-:\hh\to \rr$ by $$u^-(x) = \phi(q) +w^-(x)-
\varepsilon,$$ where $w^-$ is given by Proposition \ref{barinf}, we
again obtain by the Comparison Principle that $u^-\le v$ for all $v\in
S^\phi,$ which gives us that $\liminf_{x\to q}(u-\Hscap_o)(x) \ge
\phi(q)-\varepsilon.$

Therefore $u$ continuously extends up to $\mathcal{B}_H$ and its
extension $\bar{u}$ satisfies $\bar{u}|_{\mathcal{B}_H} = \phi$.

\end{proof}

\end{document}